\pdfoutput=1
\documentclass{amsart}

\usepackage{amssymb}

\usepackage{graphicx}

\usepackage{mathrsfs}
\usepackage{amsfonts}

\usepackage{color}
\usepackage{verbatim}

\newtheorem{theorem}{Theorem}[section]
\newtheorem*{theorem*}{Theorem}

\newtheorem{lemma}[theorem]{Lemma}
\newtheorem{proposition}[theorem]{Proposition}
\newtheorem{corollary}[theorem]{Corollary}
\newtheorem{conjecture}[theorem]{Conjecture}

\theoremstyle{definition}
\newtheorem{definition}[theorem]{Definition}
\newtheorem{example}[theorem]{Example}

\theoremstyle{remark}

\numberwithin{equation}{section}

\begin{document}

\title{$k$-Parabolic Subspace Arrangements}

\author{H\'{e}l\`{e}ne Barcelo}

\address{School of Mathematical and Statistical Sciences, Arizona State University, Tempe, AZ and 
Mathematical Sciences Research Institute, Berkeley, CA}
\email{hbarcelo@msri.org}

\author{Christopher Severs}
\address{Mathematical Sciences Research Institute, Berkeley, CA}
\email{csevers@msri.org}

\author{Jacob A. White}
\address{School of Mathematical and Statistical Sciences, Arizona State University, Tempe, AZ}
\email{jacob.a.white@asu.edu}


\date{}


\begin{abstract}

In this paper, we study $k$-parabolic arrangements, a generalization of $k$-equal arrangements for finite real reflection groups. 
When $k=2$, these arrangements correspond to the well-studied Coxeter arrangements. 
Brieskorn (1971) showed that the fundamental group of the complement, over $\mathbb{C}$, of the type $W$ Coxeter arrangement is isomorphic to the pure Artin group of type $W$. 
Khovanov (1996) gave an algebraic description for the fundamental group of the complement, over $\mathbb{R}$, of the $3$-equal arrangement. 
We generalize Khovanov's result to obtain an algebraic description of the fundamental groups of the complements 
of $3$-parabolic arrangements for arbitrary finite reflection groups. 
Our description is a real analogue to Brieskorn's description.

\end{abstract}

\maketitle

\section{Introduction}

A \emph{subspace arrangement} $\mathscr{A}$ is a collection of linear subspaces of a finite dimensional vector space $V$, 
such that there are no proper containments among the subspaces.  
One of the main questions regarding subspace arrangements is to study the structure of the 
\emph{complement} $\mathcal{M}(\mathscr{A}) = V - \cup_{X \in \mathscr{A}} X$. 
The original motivation for studying this structure comes from the work of Arnold \cite{arnold}. He studied the topology of the complement of
the Braid arrangement, defined below, in connection to his work on the Thirteenth Hilbert problem. A history detailing his work in this area, and related 
results on studying the topology of complements of hyperplane arrangements can be found in the Introduction of the book \emph{Arrangements of Hyperplanes}
\cite{orlik-terao}.

Much of the information about the structure of the complement may be captured by studying its (co)homology and homotopy groups. In this paper we will be mainly concerned with the homotopy groups of the complement. We note that when we say homotopy we refer to both the classical notion of homotopy as well as the newer \emph{discrete homotopy theory} of Laubenbacher and Barcelo, et al. developed in \cite{foundations}. Discrete homotopy theory involves constructing a bigraded sequence of groups defined on an abstract simplicial complex (the Coxeter complex in this case) that are invariants of a combinatorial nature. Instead of being defined on the topological space of a geometric realization of a simplicial complex, the discrete homotopy groups are defined in terms of the \emph{combinatorial} connectivity of the complex. In fact, we show that by studying this combinatorial invariant of the subspace arrangement we recover information about the classical homotopy of the arrangement. 

We begin with some motivating examples in the study of the homotopy of the complement of a hyperplane arrangement. The first example concerns the braid arrangement, mentioned above, which is the quintessential example in the study of hyperplane arrangements. 

The braid arrangement is the collection of ``diagonals'' $z_i = z_j$ for $1 \leq i < j \leq n$ from a complex $n$-dimensional vector space. 
In 1963, Fox and Neuwirth \cite{fox-neuwirth} showed that the fundamental group of the complement is isomorphic to the pure braid group. 
It was also shown by Fadell and Neuwirth \cite{fadell-neuwirth} that the braid arrangement is a $K(\pi,1)$ arrangement. A path-connected space is said to be $K(\pi,1)$ (which is actually a special case of a $K(\pi,m)$, or Eilenberg-MacLane space) if the homotopy groups $\pi_i$ are trivial for $i>1$ and the fundamental group $\pi_1$ is isomorphic to the group $\pi$. We say that an arrangement is $K(\pi,1)$ if the complement of the arrangement is a $K(\pi,1)$ space. Determining if an arrangement is $K(\pi,1)$ is one of the classical problems in the study of hyperplane and subspace arrangements. 

Another classical example of a $K(\pi,1)$ space comes from a simplicial hyperplane arrangement. A simplicial hyperplane arrangement is an arrangement whose regions are simplicial cones. In 1972, Deligne \cite{deligne}, proved that the complement of the complexification of any simplicial hyperplane arrangement is $K(\pi,1)$. Deligne's result also proved a conjecture of Brieskorn, that the complexification of a reflection arrangement is $K(\pi,1)$. This is due to the fact that reflection arrangements are simplicial. Moreover, in 1971 Brieskorn \cite{brieskorn} found that the fundamental group of the complement of the complexification of a reflection arrangement is isomorphic to the pure Artin group of type $W$. This illustrates another central question in the study of hyperplane and, more generally, subspace arrangements. Can we give a presentation for the fundamental group of the complement in terms of generators and relations?

In the case of complex hyperplane arrangements, this question has been answered completely by Arvola. In \cite{arvola}, he provides a method to find a presentation for any complex hyperplane arrangement, not just those arrangements which are complexifications of real arrangements. Although the work of Arvola (and Deligne and Brieskorn before him) covers a large class of subspace arrangements, the subspace arrangements that we consider in this paper cannot be viewed as complex hyperplane arrangements and so the above results do not apply. We present an example of such an arrangement now. 

A real subspace arrangement which cannot be viewed as a complex hyperplane arrangement is the $3$-equal arrangement. This arrangement is the collection of all subspaces of the form $x_i = x_j = x_k$ for $1 \leq i < j < k \leq n$ in a real 
$n$-dimensional vector space. In 1996, Khovanov proved that the complement of this arrangement is $K(\pi,1)$ \cite{khovanov}. He also gave a presentation for the fundamental group of the complement. The presentation of this group, as well as the presentation of the pure braid group, use the symmetric group in their construction. It is well known that the symmetric group is generated by simple transpositions $s_i = (i, i+1)$, $i \in [n-1]$, 
subject to the following relations:

\begin{enumerate}
 \item $s_i^2 = 1$

  \item $s_i s_j = s_j s_i$, if $|i - j| > 1$

  \item $s_i s_{i+1} s_i = s_{i+1} s_i s_{i+1}$
\end{enumerate}

The braid group has presentation given by the same generating set, but subject only to relations 2 and 3. 
The pure braid group is the kernel of the surjective homomorphism, $\varphi$, from the braid group to the symmetric group, given by $\varphi(s_i) = s_i$ for all $i \in [n-1]$. Khovanov's presentation of the fundamental group of the complement of the $3$-equal arrangement is very similar. Khovanov introduces the \emph{triplet group}, which we denote $A_{n-1}'$. It is the group generated by the simple transpositions $s_i$, but subject only to relations 1 and 2. He defines the \emph{pure triplet group} to be the kernel of the surjective homomorphism, 
$\varphi': A_{n-1}' \to A_{n-1}$, given by $\varphi'(s_i) = s_i$ for all $i \in [n-1]$.  Khovanov then shows that the fundamental group of the complement of the $3$-equal arrangement is isomorphic to the pure triplet group \cite{khovanov}. This provides a real analogue to the results of Fadell, Fox and Neuwirth.

In the same spirit as Khovanov, we give real analogues of the results of Deligne and Brieskorn for subspace arrangements in $\mathbb{R}^n$ 
that correspond to finite real reflection groups. In particular, given a finite real reflection group $W$, we define a family of (real) subspace arrangements which we call $k$-parabolic arrangements, denoted $\mathscr{W}_{n,k}$.  We define an analogue of the pure Artin group for a real Coxeter arrangement in the following way. Construct a new Coxeter group $W'$ on the same generating set $S$ as $W$, but we relax all relations of $W$ that are not commutative relations or square relations. Then the following is true:

\begin{theorem*}[Theorem 4.1]
 The fundamental group of the complement of the $3$-parabolic arrangement $\mathscr{W}_{n,3}$ is isomorphic to the kernel of a surjective homomorphism $\varphi': W' \to W$ given by $\varphi'(s) = s$ for all $s \in S$. 
\end{theorem*}

Our presentation for the fundamental group of such arrangements uses the discrete homotopy group of the arrangement. We show that the discrete fundamental group of the Coxeter complex is isomorphic to $\pi_1$ of the complement of the $3$-parabolic arrangement. Thus, our result shows that sometimes we can replace a group defined in terms of the topology of the space with a group defined in terms of the combinatorial structure of the space. 

It is also true that $\mathscr{W}_{n,3}$ is a $K(\pi,1)$ arrangement. $\mathscr{W}_{n,3}$ is a collection of codimension 2 subspaces that are invariant under the action of some finite real reflection group $W$ and by a result of Davis, Januszkiewicz and Scott \cite{blowup} such arrangements are in fact $K(\pi,1)$.

Although this paper deals primarily with the homotopy of the complement of a subspace arrangement, we would be remiss if we did not mention two closely related results about the cohomology groups of the complement of a subspace arrangement and the homotopy groups of the \emph{union} of the subspace arrangement. The most well known result about the cohomology groups of the complement is due to Goresky and MacPherson. In \cite{goresky-macpherson} they provide a formula that establishes an ismorphism between the cohomology group of the complement and a direct sum of the homology groups of the partially ordered set (poset) of intersections of the subspace arrangement, $\mathcal{L}(\mathscr{A})$. In practice, this result allows one to compute the cohomology groups of the complement of a subspace arrangement by studying the order complexes of intervals of the poset $\mathcal{L}(\mathscr{A})$. A related formula for the homotopy groups of the union of a subspace arrangement is given by Ziegler and \v{Z}ivaljevi\'c in \cite{zieg-ziv}. While the Goresky-MacPherson formula certainly provides a method to calculate the Betti numbers of the complement of a subspace arrangement, it is sometimes not easy to implement. Barcelo and Smith use discrete homotopy theory to give a fomula to compute the first Betti number of the complement of the 3-equal arrangement in \cite{barcelo-smith} without resorting to the Goresky-MacPherson formula. Our results about the fundamental group of the $3$-parabolic arrangement may be used in a similar manner to calculate the first Betti number of the arrangement, though we do not pursue this direction in the current paper. 

The remainder of the paper is organized as follows. 

In Section 2 we give a definition of the $k$-parabolic arrangement, and review some necessary definitions related to Coxeter groups. 
We also relate $k$-parabolic arrangements to previous analogues of the $k$-equal arrangement given by Bj\"orner and Sagan 
\cite{subspacesBD} for types $B$ and $D$. In Section 3, we give a brief overview of discrete homotopy theory 
and the definition of the Coxeter complex. Then we give an isomorphism between the classical fundamental group of the complement 
of the $3$-parabolic arrangement and the discrete fundamental group of the corresponding Coxeter complex. 
In Section 4, we use this isomorphism and a study of discrete homotopy loops in the Coxeter complex to obtain our algebraic description 
of the fundamental group of the complement of the $3$-parabolic arrangement. In Section 5, we study the discrete fundamental groups 
of the Coxeter complex for $k > 3$, and show that they are trivial. 
In Section 6 we conclude with some open questions related to $\mathscr{W}_{n,k}$-arrangements as well as a discussion on the $K(\pi,1)$ 
problem.

\section{Definition of the $\mathcal{W}_{n,k}$-arrangement} 

Let $W$ be a finite real reflection group acting on $\mathbb{R}^n$ 
and fix a root system $\Phi$ associated to $W$. Let  $\Pi$ $\subset \Phi$ 
be a fixed simple system. Finally, let $S$ be the set of simple reflections associated to $\Pi$. 
Assume that $\Pi$ spans $\mathbb{R}^n$. We let $m(s,t)$ denote the order of $st$ in $W$. 
We know that $m(s,s) = 1$ and $m(s,t) = m(t,s)$ for all $s,t \in S$. 
Finally, given a root $\alpha$, let $s_{\alpha}$ denote the corresponding reflection, 
and let $(\cdot, \cdot)$ denote the standard inner product.

Recall that there is a hyperplane arrangement associated to $W$, 
called the \emph{Coxeter arrangement} $\mathscr{H}(W)$, which consists of hyperplanes 
$ H_{\alpha} = \{ x \in \mathbb{R}^n: (x,\alpha) = 0 \}$ for each $\alpha \in \Phi^+$.
Since $\Pi$ spans $\mathbb{R}^n$, the Coxeter arrangement is central and essential, 
which implies that the intersection of all the hyperplanes is the origin. 

There is a poset associated to $\mathscr{H}(W)$. For the Coxeter Arrangement, the associated poset is $\mathcal{L}(\mathscr{H}(W))$, the poset of all intersections of hyperplanes, ordered by reverse inclusion. 
In fact, every pair of elements in this poset have a unique upper bound and lower bound, and hence this poset is actually a lattice.

Since we are generalizing the $k$-equal arrangement, which corresponds to the case $W = A_n$, 
we use it as our motivation. For this paper, we will actually work with the essentialized $k$-equal arrangement. 
The $k$-equal arrangement, $\mathscr{A}_{n,k}$, is the collection of all subspaces given by 
$x_{i_1} = x_{i_2} = \ldots = x_{i_k}$ over all indices $\{i_1, \ldots, i_k\} \subset [n+1]$, 
with the relation $\sum_1^{n+1}x_i = 0$.  We note that the intersection poset 
$\mathcal{L}(\mathscr{A}_{n,k})$ is a subposet of $\mathcal{L}(\mathscr{H}(A_n))$. 
There is already a well-known combinatorial description of both of these posets. 
The poset of all set partitions of $[n+1]$ ordered by refinement is isomorphic to $\mathcal{L}(\mathscr{H}(A_n))$, 
and under this isomorphism, $\mathcal{L}(\mathscr{A}_{n,k})$ is the subposet of set partitions 
where each block is either a singleton, or has size at least $k$. 
However, our generalization relies on a lesser-known description of these posets in terms of parabolic subgroups.

\begin{definition}

A subgroup $G \subseteq W$ is a \emph{parabolic subgroup} if there exists 
a subset $T \subseteq S$ of simple reflections, and an element $w \in W$ such that $G = <wTw^{-1}>$. 
If $w$ can be taken to be the identity, then $G$ is a \emph{standard} parabolic subgroup. 
We view $(G, wTw^{-1})$ as a Coxeter system, and call $G$ \emph{irreducible} if $(G, wTw^{-1})$ is an irreducible system. 

\end{definition}

It is well known that the lattice of \emph{standard} parabolic subgroups, ordered by inclusion, 
is isomorphic to the Boolean lattice. However, the lattice of \emph{all} parabolic subgroups, 
$\mathscr{P}(W)$, ordered by inclusion, was shown by Barcelo and Ihrig \cite{barcelo-ihrig} to be isomorphic 
to $\mathcal{L}(\mathscr{H}(W))$. Since this isomorphism is essential to our generalization, 
we review it. The isomorphism is given by sending a parabolic subgroup 
$G$ to $Fix(G) = \{ x \in \mathbb{R}^n: wx = x, \forall w \in G \}$, and the inverse is given by 
sending an intersection of hyperplanes $X$ to $Gal(X) = \{w \in W: wx = x, \forall x \in X \}$. 

\begin{figure}[htbp]
\includegraphics[height=3.9cm]{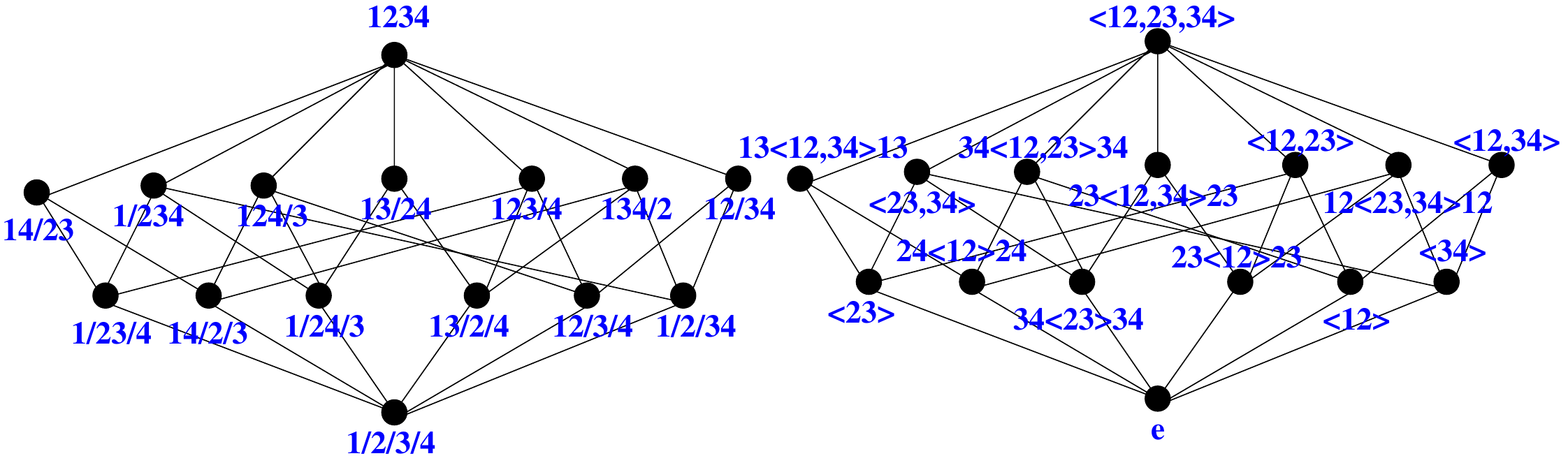}

\caption{The Galois Correspondence for $A_3$}

\end{figure}

This Galois correspondence gives a description of $\mathcal{L}(\mathscr{H}(A_n))$ in terms of parabolic subgroups of $A_n$. 
We also obtain another description of $\mathcal{L}(\mathscr{A}_{n,k})$ under this correspondence.

\begin{proposition}

The Galois correspondence gives a bijection between subspaces of $\mathscr{A}_{n,k}$ and irreducible parabolic subgroups of $A_n$ 
of rank $k-1$.

\end{proposition}

\begin{proof}
Let $X$ be a subspace of $\mathbb{R}^{n+1}$ given by $x_1 = \ldots = x_k$. 
The $k$-equal arrangement is the orbit of $X$ under the action of $A_n = S_{n+1}$, 
and $Gal(X) = <(1,2), ..., (k-1,k)>$, hence is irreducible. For $w \in A_{n}$, $Gal(wX) = wGal(X)w^{-1}$, 
so all of the subspaces in the $k$-equal arrangement have irreducible Galois groups. 

Conversely, every irreducible parabolic subgroup of rank $k-1$ in $A_{n}$ is the Galois group 
of some subspace in the $k$-equal arrangement. To see this, consider an irreducible parabolic subgroup 
$G$ of rank $k-1$. Then there exists a standard parabolic subgroup $H$ and an element $w \in W$ 
such that $G = wHw^{-1}$. Since $H$ is an irreducible \emph{standard} parabolic subgroup, 
$H = <(i, i+1), ..., (i+k-1, i+k)>$ for some $1 \leq i \leq n+1-k$. Thus, $Fix(H)$ is given by 
$x_i = \ldots = x_k$, and $Fix(G) = Fix(wHw^{-1}) = wFix(G)$ is given by $x_{w(i)} = \ldots x_{w(k)}$, 
which is a subspace in the $k$-equal arrangement. \end{proof}

With this proposition as motivation, we give the following definition for a $k$-parabolic arrangement.

\begin{definition}
Let $W$ be an finite real reflection group of rank $n$.
Let  $\mathscr{P}_{n,k}(W)$ be the collection of \emph{all} irreducible parabolic subgroups of $W$ of rank $k-1$.

Then the $k$-parabolic arrangement $\mathscr{W}_{n,k}$ is the collection of subspaces 

$$\{ Fix(G): G \in \mathscr{P}_{n,k}(W) \}.$$

\end{definition}

The $k$-parabolic arrangements have many properties in common with the $k$-equal arrangements. 
Both of these arrangements can be embedded in the corresponding Coxeter arrangement. 
That is, every subspace in these arrangements can be given by intersections of hyperplanes of the Coxeter arrangements. 
Moreover, $\mathcal{L}(\mathscr{W}_{n,k})$ is a subposet of 
$\mathcal{L}(\mathscr{H}(W))=\mathcal{L}(\mathscr{W}_{n,2})$, and these arrangements are invariant under the action of $W$. 
Indeed, consider a subspace $X$ in $\mathscr{W}_{n,k}$ and an element $w \in W$. 
Since $X$ is in $\mathscr{W}_{n,k}$, $Gal(X)$ is an irreducble parabolic subgroup of rank $k-1$. 
It is clear that $Gal(wX) = wGal(X)w^{-1}$, so $Gal(wX)$ is also an irreducible parabolic subgroup of rank $k-1$, 
whence $Gal(wX) \in \mathscr{P}_{n,k}(W)$. Since $Fix(Gal(wX)) = wX$, it follows that $wX \in \mathcal{W}_{n,k}$.

When $W$ is of type $A$, we see that we have recovered the $k$-equal arrangement. 
To see what happens when $W$ is type $B$ or $D$, first we recall type $B$ and $D$ 
analogues of the $k$-equal arrangement. In 1996, Bj\"orner and Sagan defined a class of subspace arrangements 
of type $B$ and $D$ \cite{subspacesBD}, which they call the $\mathscr{B}_{n,k,h}$-arrangements and $\mathscr{D}_{n,k}$-arrangements. 

\begin{definition}
The $\mathscr{D}_{n,k}$-arrangement consists of subspaces  given by 
$\pm x_{i_1} = \pm x_{i_2} = \ldots = \pm x_{i_k}$, over distinct indices 
$i_1, \ldots, i_k$. The $\mathscr{B}_{n,k,h}$-arrangements are obtained from the $\mathscr{D}_{n,k}$-arrangements 
by including subspaces given by $x_{i_1} = \ldots = x_{i_h} = 0$ over distinct indices $i_1, \ldots, i_h$, with $h < k$. 
\end{definition}

The Betti numbers of $\mathcal{M}(\mathscr{B}_{n,k,h})$ were computed by Bj\"orner and Sagan in \cite{subspacesBD}, 
while the Betti numbers of $\mathcal{M}(\mathscr{D}_{n,k})$ were computed by Kozlov and Feichtner in \cite{subspacesD}.

\begin{example}[When $W$ is of type $B$] 
When $W$ is of type $B$, the $k$-parabolic arrangement is the $\mathscr{B}_{n,k,k-1}$-arrangement of Bj\"orner and Sagan \cite{subspacesBD}. 
Recall that $B_n$ has presentation given by generators $s_i, 0 \leq i \leq n$, such that 
$<s_1, \ldots, s_n>$ generate the symmetric group, $(s_0s_1)^4 = 1$, and $s_0s_i = s_is_0$ for $i > 1$. 
It is well-known that the $\mathscr{B}_{n,k,k-1}$-arrangement is the orbit of two subspaces given by 
$x_1 = \ldots = x_k$ and $x_1 = \ldots = x_{k-1} = 0$, under the action of $B_n$. 
Clearly the Galois groups of these two spaces are given by $<s_1, \ldots, s_{k-1}>$ and $<s_0, \ldots, s_{k-2}>$. 
These are both irreducible parabolic subgroups of rank $k-1$, so every subspace of the 
$\mathscr{B}_{n,k,k-1}$-arrangement corresponds to an irreducible parabolic subgroup of rank $k-1$. 
Similarly, given an irreducible parabolic subgroup of rank $k-1$, 
it is not hard to show that this subgroup corresponds to a subspace in the $\mathscr{B}_{n,k,k-1}$-arrangement. 
The argument is similar to the case for type $A$, and we omit the details.
\end{example}

\begin{example}[When $W$ is of type $D$]
When $W$ is of type $D$, the $3$-parabolic arrangement is the $\mathscr{D}_{n,3}$-arrangement. 
However, when $W = D_n$ and $k>3$, the $k$-parabolic arrangement is not the same as the 
$\mathscr{D}_{n,k}$-arrangement as defined by Bj\"orner and Sagan. 
If we fix the standard basis of $\mathbb{R}^n$ $e_1, \ldots, e_n$, then $D_n$ has a simple root system 
consisting of $\alpha_1 = e_1+e_2, \alpha_i = e_{i-1}-e_i$, $2 \leq i \leq n$. 
We see that the subgroup $G$ generated by the simple reflections associated to $\alpha_1, \ldots, \alpha_{k-1}$ 
is a standard irreducible parabolic subgroup of rank $k-1$, and is in fact a Coxeter system for $D_k$. 
Also, $Fix(G)$ is given by $x_1 = x_2 = \ldots = x_{k-1} = 0$, 
which is not a subspace of the $\mathscr{D}_{n,k}$-arrangement defined by Bj\"orner and Sagan. 
Since $\mathscr{W}_{n,k}$ is closed under the action of $W$, it is clear that $\mathscr{W}_{n,k}$ 
includes subspaces of the form $x_{i_1} = ... = x_{i_{k-1}} = 0$, over indices $1 \leq i_1 < \ldots < i_k \leq n$, 
so in fact the $k$-parabolic arrangement corresponds to the $\mathscr{B}_{n,k,k-1}$-arrangement of Bj\"orner and Sagan, 
and thus, the $k$-parabolic arrangements of type $B$ and $D$ are the same when $k > 3$. 
\end{example}

We end this section with the remark that it is possible to redefine 
the $\mathscr{D}_{n,k}$ and $\mathscr{B}_{n,k,h}$ arrangements in a way similar to our definition 
of $k$-parabolic arrangements. Given an irreducible parabolic subgroup 
$G = <wIw^{-1}>$ for some $I \subseteq S, w \in W$, the \emph{type} of $G$ is the type of the Coxeter system 
$(G, wIw^{-1})$ in the Cartan-Killing classification. 
Then the $\mathscr{D}_{n,k}$-arrangement is the set of all $Fix(G)$, 
where $G$ is a type $A$ irreducible parabolic subgroup of $D_n$ of rank $k-1$. 
Similarly, the $\mathscr{B}_{n,k,h}$-arrangement is the set of all $Fix(G)$, 
where $G$ is either a type $A$ irreducible parabolic sugroup of rank $k-1$, or $G$ is a type $B$ 
irreducible parabolic subgroup of rank $h$.

\section{Discrete Homotopy Theory}

In this section, we recall the definitions of discrete homotopy group, $A_1^q$, on a simplicial complex, and then apply the theory to the Coxeter complex of type $W$. We obtain an 
isomorphism between the fundamental group of $\mathcal{M}(\mathscr{W}_{n,3})$ and $A_1^{n-2}$ of the Coxeter complex. This isomorphism will be used in the next section to obtain our 
presentation of the fundamental group of $\mathcal{M}(\mathscr{W}_{n,3})$.

Recall that an abstract simplicial complex on a set $X$ is a collection $\Delta$ of subsets of $X$, such that:
\begin{enumerate}
 \item $\{x\} \in \Delta, \mbox{ for all } x \in X$
 \item if $S \subseteq T \in \Delta$, then $S \in \Delta$.
\end{enumerate}

The following theorem motivates our result for the case $W = A_n$.

\begin{theorem}
\label{Anresult}
Let $\mathcal{M}(\mathscr{A}_{n,k})$ be the complement of the $k$-equal arrangement $\mathscr{A}_{n,k}$. 
Let $\mathscr{C}(A_n)$ be the order complex of the Boolean lattice.

Then $\pi_1(\mathcal{M}(\mathscr{A}_{n,k})) \cong A_1^{n-k+1}(\mathscr{C}(A_n))$, 
where $A_1^q$ is a discrete homotopy group, to be defined below.

\end{theorem}

This result was shown independently by Bj\"orner \cite{bjorner-atheory} and Babson (appears in \cite{perspectives}) in 2001). It turns out that the order complex of the Boolean lattice 
is the Coxeter complex of type $A$, which explains our choice of notation.

One of the original motivations for discrete homotopy theory was to create a sequence of groups 
for studying social networks being modeled as simplicial complexes. 
As Theorem \ref{Anresult} shows, 
discrete homotopy theory has applications in studying simplicial complexes arising from geometry as well a social networks. 
We will show that there is an isomorphism between $\pi_1(\mathcal{M}(\mathscr{W}_{n,k}))$ 
and the discrete fundamental group, $A_1^{n-k+1}$, of the Coxeter complex, 
a combinatorial structure associated to the Coxeter arrangement. 
Essentially, we are replacing a topologically defined group with a combinatorially defined group. 
First, however, we give an overview of some of the needed basic definitions and results from discrete homotopy theory. 
Many details and background history of discrete homotopy theory can be found in \cite{foundations}, 
while a categorical approach is given in \cite{graph-homotopy}.

Fix a positive integer $d$. Let $\Delta$ be a simplicial complex of dimension $d$, 
fix $0\le q \le d$, and let $\sigma_0 \in \Delta$ be maximal with dimension $\geq q$. 
Two simplicies $\sigma$ and $\tau$ are \emph{$q$-near} if they share $q+1$ elements. 
A \emph{$q$-chain} is a sequence $\sigma_1, \ldots, \sigma_k$, such that 
$\sigma_i, \sigma_{i+1}$ are $q$-near for all $i$. A \emph{$q$-loop} based at $\sigma_0$ is a $q$-chain 
with $\sigma_1 = \sigma_k = \sigma_0$.

\begin{figure}[htbp]
  \center{\includegraphics[height= 4cm]{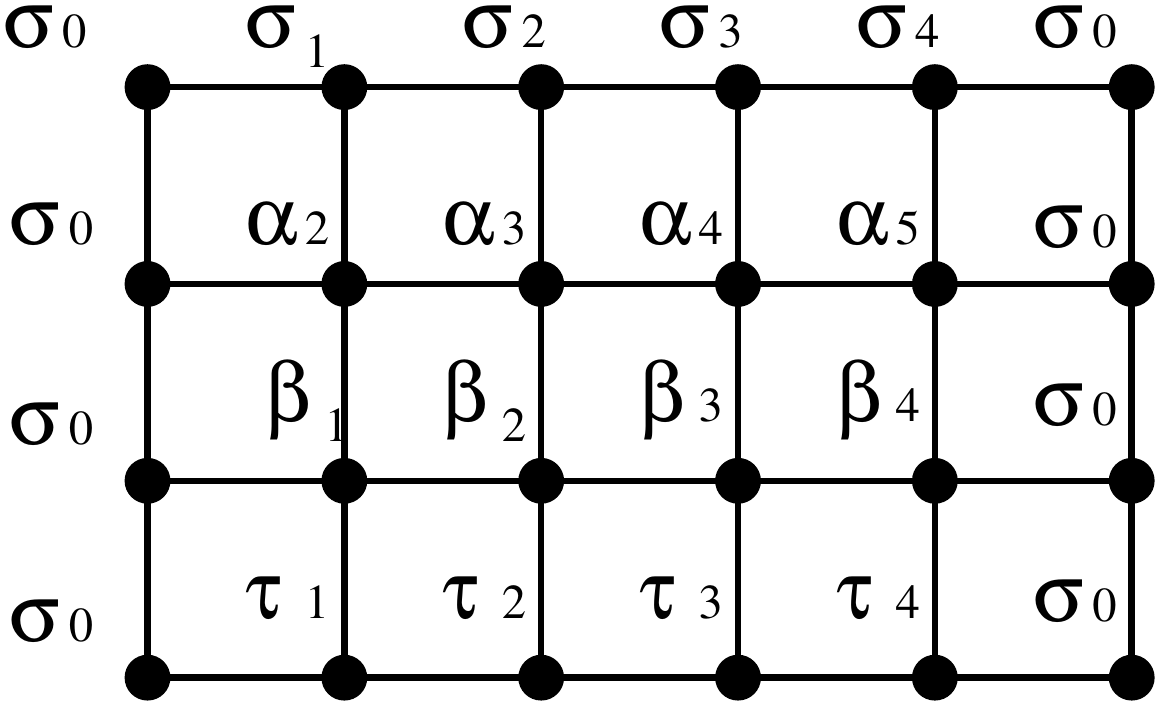}}
  \caption{An example of a homotopy grid}
\label{fig:grid}
\end{figure}

\begin{definition}
We define an equivalence relation, $\simeq_q$ on $q$-loops with the following conditions: 

\begin{enumerate}
\item The $q$-loop
\begin{equation*}
(\sigma)=(\sigma_0, \sigma_1, \ldots,\sigma_i, \sigma_{i+1},\ldots, \sigma_n, \sigma_0)
\end{equation*}
is equivalent to the $q$-loop
\begin{equation*}
(\sigma)'=(\sigma_0, \sigma_1, \ldots, \sigma_i, \sigma_i, \sigma_{i+1}, \ldots, \sigma_n, \sigma_0),
\end{equation*}
 which we refer to as stretching.

\item If $(\sigma)$ and $(\tau)$ are two $q$-loops that have the same length then they are equivalent 
if there is a diagram as in figure \ref{fig:grid}. The vertices represent simplices, 
and two vertices are connected by an edge if and only if the corresponding simplices are $q$-near. 
Thus, every row is a $q$-loop based at $\sigma_0$, and every column is a $q$-chain. 
Such a diagram is called a (discrete) homotopy between $(\sigma)$ and $(\tau)$.
\end{enumerate}
\end{definition}

When there is no possibility for confusion, we will use $(\sigma) \simeq (\tau)$ to mean $(\sigma) \simeq_q (\tau)$. 
Also we may use the term $q$-homotopy when necessary to avoid confusion (in Section 5 this will be highly necessary).

Define $A_1^q(\Delta, \sigma_0)$ to be the collection of equivalence classes of $q$-loops based at $\sigma_0$. 
Then the operation of concatenation of $q$-loops gives a group operation on $A_1^q(\Delta, \sigma_0)$, 
the \emph{discrete homotopy group} of $\Delta$. The identity is the equivalence class containing the trivial loop 
$(\sigma_0)$, and given an equivalence class $[\sigma]$ for the 
$q$-loop $(\sigma) = (\sigma_0, \sigma_1, \ldots, \sigma_k, \sigma_0)$, 
the inverse $[\sigma]^{-1}$ is the equivalence class of 
$(\sigma_0, \sigma_k, \sigma_{k-1}, \ldots, \sigma_2, \sigma_1, \sigma_0)$. 
As in classical topology, if a pair of maximal simplices $\sigma$, $\tau$ of dimension at least $q$ 
in $\Delta$ are $q$-connected, then $A_1^q(\Delta, \sigma) \cong A_1^q(\Delta, \tau)$. 
Thus, in the case $\Delta$ is $q$-connected, we will set $A_1^q(\Delta) = A_1^q(\Delta, \sigma_0)$ 
for any maximal simplex $\sigma_0 \in \Delta$ of dimension at least q. 

Before we use discrete homotopy theory, 
we need a result from \cite{foundations} that relates discrete homotopy theory of a simplicial complex 
to classical homotopy theory of a related space. Given $0 \le q \le d$, let $\Gamma^q(\Delta)$ be a graph 
whose vertices are maximal simplices of $\Delta$ of size at least $q$, and with edges between two simplices 
$\sigma$, $\tau$, if and only if $\sigma$ and $\tau$ are $q$-near. 
Then the following result relates $A_1^q(\Delta, \sigma_0)$ in terms of a cell complex related to $\Gamma^q(\Delta)$.

\begin{proposition}[Proposition 5.12 in \cite{foundations}]
\label{comptograph}
\begin{equation*}
A_1^q(\Delta, \sigma_0) \cong \pi_1(X_{\Gamma}, \sigma_0)
\end{equation*}
where $X_{\Gamma}$ is a cell complex obtained by gluing a $2$-cell on each $3$- and $4$-cycle of $\Gamma=\Gamma^q(\Delta)$.
\end{proposition}

We now briefly review the definition of a Coxeter complex and the associated polytope call a $W$-Permutahedron.

Let $W, \mathcal{H}(W), \Phi, \Pi, S$ be as in section 2. 
As mentioned previously, we study the discrete homotopy groups of the \emph{Coxeter complex} 
associated to $W$, and relate them to $\pi_1(\mathcal{M}(\mathscr{W}_{n,k}))$. The majority of these details 
can be found in Section 1.14 in Humphrey's book on Coxeter groups \cite{coxbook}. The concepts regarding fans and zonotopes 
can be found in Chapter 7 of Zeigler's book on polytopes \cite{polybook}.

For a given set $I \subseteq S$, let $W_I = <I>$, and $\Pi_I = \{ \alpha \in \Pi: s_{\alpha} \in I \}$. 
We can associate to $W_I$ the set of points 
$C_I = \{ x \in \mathbb{R}^n: (x,\alpha) = 0, \forall \alpha \in \Pi_I, \mbox{ and } (x, \alpha) > 0, \forall \alpha \in \Pi - \Pi_I \}$.
 The set $C_I$ is the intersection of hyperplanes $H_{\alpha}$ for $\alpha \in \Pi_I$
 with certain open half-spaces. We see that $C_{\emptyset}$ corresponds to the interior of a fundamental region, 
and $C_{S}$ is the origin.

For a given coset $wW_I$, we can associate the set of points $wC_I$. 
The collection $\mathscr{C}(W)$ of $wC_I$ for all $w \in W,$ and all $ I \subseteq S$ partitions 
$\mathbb{R}^n$, and is called the \emph{Coxeter complex} of $W$. The face poset of the Coxeter complex 
can be viewed as the collection of cosets $wW_I$ for any $w \in W, I \subseteq S$, ordered by \emph{reverse inclusion}. 
We note that this poset is \emph{not} $\mathcal{L}(\mathscr{H}(W))$. 
 For $W = A_n$, we have already mentioned that $\mathcal{L}(\mathscr{H}(W))$ is isomorphic to the partition lattice. 
The face poset of the braid arrangement, however, is isomorphic to the order complex of the boolean lattice. 
Since chains in the boolean lattice are in one-to-one correspondence with 
\emph{ordered} set partitions, these two posets are related, but are very different.

\begin{figure}[htbp]
  \center{\scalebox{.7}{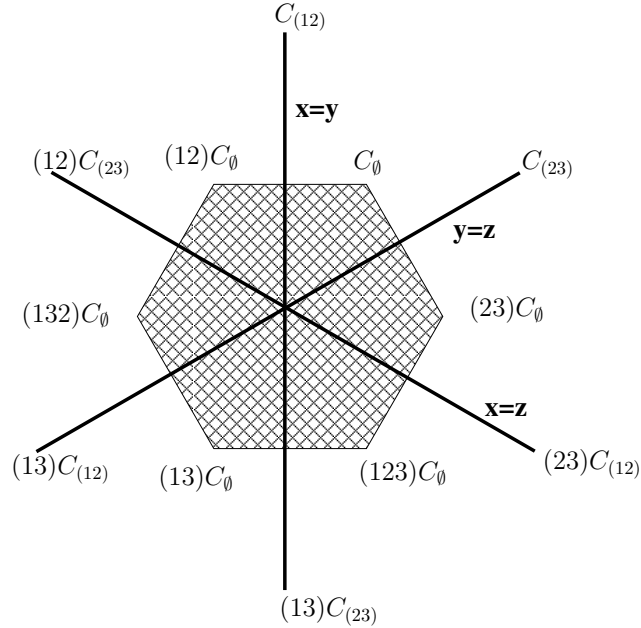}}
  \caption{Coxeter Complex and Zonotope for $W = A_3$. Note that $C_S$ is the origin.}
\end{figure}

For a given $w, I$, the closure of $wC_I$ is a convex polyhedral cone. 
In fact, the collection of all $w\bar{C}_I$ forms a fan of $\mathbb{R}^n$, 
which is the fan associated to $\mathscr{H}(W)$. Under this view, the sets $w\bar{C}_I$ are the faces of the arrangement.

Recall that we can associate a \emph{zonotope} to a hyperplane arrangement. 
That is, given line segments of unit length normal to the hyperplanes, one can form a polytope 
by taking the Minkowski sum of these line segments. For a Coxeter arrangement of type $W$ 
this zonotope is called the $W$-Permutahedron. Also, the fan of the arrangement is the normal fan of the zonotope. 
Thus, we can label the faces of the $W$-Permutahedron by cosets $wW_I$, 
where a face $F$ gets the label $wW_I$ if the normal cone for $F$ is $w\bar{C}_I$. 
Under this labeling, the face poset of the $W$-Permutahedron is indexed by cosets
 $wW_I$ for all $w \in W, I \subseteq S$, ordered by \emph{inclusion}. 

We observe that in the $W$-Permutahedron, the vertices correspond to elements of $W$, 
and two vertices share an edge if and only if the corresponding regions share an $(n-1)$-dimensional boundary, 
that is if and only if the corresponding elements of $W$ differ by multiplication on the right by a simple reflection. 
From this it follows that the graph (one-skeleton) of the $W$-Permutahedron is the graph 
$\Gamma^{n-2}(\mathscr{C}(W))$ defined before Proposition \ref{comptograph}.

We also characterize the cycles that are boundaries of $2$-faces in the $W$-Permutahedron. 
Given a $2$-dimensional face $F$ and a vertex $w$ in $F$, we see that one edge adjacent 
to $w$ in $F$ is of the form $w, ws$ for some $s \in S$. Likewise, one of the two edges of $F$ incident 
to the edge $w, ws$ is the edge $ws, wst$, where $t \in S - s$. Thus we see that the coset associated 
to the normal cone of $F$ contains both $wW_{s}$ and $wsW_{t}$. Likewise, it is the smallest coset to 
contain these two cosets, so the corresponding coset is given by $wW_{\{s,t\}}$. The cycle that is the 
boundary of $F$ is seen to have length $2m(s,t)$. This means that the graph has no 3-cycles, and 4-cycles 
are boundaries of faces which correspond to a coset of $W_{\{s, t \}}$, where $s, t \in S$ and $m(s,t) = 2$. 
The fact that the graph has no $3$-cycles will turn out to be useful in section 4.

Now we turn to the main result of this section.

\begin{theorem}
\label{mainresult}
Let $\mathcal{M}(\mathscr{W}_{n,k})$ be the complement of the $k$-parabolic arrangement $\mathscr{W}_{n,k}$.

Then $\pi_1(\mathcal{M}(\mathscr{W}_{n,k})) \cong A_1^{n-k+1}(\mathscr{C}(W))$.

\end{theorem}

We need the following proposition, which is a special case of Proposition 3.1 from Bj\"orner and Ziegler in \cite{bjorner-ziegler}.

\begin{proposition}

Let $\Delta$ be a simplicial decomposition of the $k$-sphere, and let $\Delta_0$ be a subcomplex of $\Delta$. 
Let $P$ be the face poset of $\Delta$, and let $P_0$ be the (lower) order ideal generated by $\Delta_0$. 
Then $|\Delta| \backslash |\Delta_0|$ is homotopy equivalent to a CW complex $X$, 
and moreover, the face poset of $X$ is $(P\backslash P_0)^*$, where $*$ denotes taking the dual poset.

\end{proposition}

\begin{proof}[Proof of Theorem \ref{mainresult}, when $k=3$]

A proof of the case when $k > 3$ will be the main topic of Section 5.

We can identify the Coxeter complex $\mathscr{C}(W)$ as the simplicial decomposition of the $(n-1)$-sphere $\mathbb{S}^{n-1}$ induced by
intersecting the sphere with the Coxeter arrangement for $W$. 
Let $\Delta_0$ be the subcomplex given by $\mathscr{W}_{n,3} \cap \mathbb{S}^{n-1}$. 
It is known that there is a strong deformation retract $f: \mathbb{R}^n \setminus \{0\} \to \mathbb{S}^{n-1}$ 
given by $f(x) = \frac{x}{|x|}$, and under the restriction of $f$ to $\mathcal{M}(\mathscr{W}_{n,3})$, 
we get a deformation retract $f': \mathcal{M}(\mathscr{W}_{n,3}) \to |\mathscr{C}(W)| \backslash |\Delta_0|$. 
By the above proposition, there is a CW complex $X$ that is homotopy equivalent to $|\mathscr{C}(W)| \backslash |\Delta_0|$. 
We intend to show that $2$-skeleton of $X$ is obtained from $\Gamma^{n-2}(\mathscr{C}(W))$ by attaching 2-cells to the $4$-cycles. 
To do so, we must study the maximal elements of $\Delta_0$, which correspond to the maximal dimensional cones that are
contained in the $\mathscr{W}_{n,3}$-arrangement. 
We claim that they correspond to cosets of standard irreducible parabolic subgroups of rank $2$.

First, we show that a cone $wC_I \cap \mathbb{S}^{n-1}$ is contained in $\Delta_0$ if and only if $C_I \cap \mathbb{S}^{n-1}$ 
is contained in $\Delta_0$, for all $I \subseteq S$.
 Clearly, if there is an $X \in \mathscr{W}_{n,3}$ such that $wC_I \subset X$, then $C_I \subset w^{-1}X \in \mathscr{W}_{n,3}$. 
The other direction follows similarly. 
So it suffices to study standard parabolic subgroups.

Now let $s,t \in S$ be distinct reflections. 
We will write $W_{s,t}$ for $W_I$ whenever $I = \{s,t\}$, and similarly we write $C_{s,t}$ for $C_I$ when $I = \{s,t\}$. 
We claim that for any $I \subseteq S$, $C_{s,t} \subseteq Fix(W_I)$ if and only if $s,t \in I$. 
Suppose $s,t \in I$, $x \in C_{s,t}$, and let $\alpha_s, \alpha_t$ be the roots of $\Pi$ corresponding to $s, t$. 
Since $x \in C_{s,t}$, we have $(x, \alpha_s) = (x, \alpha_t) = 0$, and it follows that $sx = tx = x$, whence $x \in Fix(W_I)$. 
For the reverse implication, without loss of generality, assume $s \not\in I$. Then $(x, \alpha_s) > 0$, 
and thus $sx \neq x$, so $x \not\in Fix(W_I)$. Hence $C_{s,t} \subseteq Fix(W_I)$ if and only if $s,t \in I$.

Let $s,t \in S$ be distinct reflections. 
If $m(s,t) > 2$, then $C_{s,t} \subset Fix(W_{s,t}) \in \mathscr{W}_{n,3}$, 
and if $m(s,t) = 2$, then for all $X \in \mathscr{W}_{n,3}$, $C_{s,t} \not\subset X$. 
Thus the maximal faces of $\Delta_0$ correspond to cosets $wW_{s,t}$, where $s,t \in S, m(s,t) > 2$, cosets of standard irreducible 
parabolic subgroups of rank 2.

Now we show that the 2-skeleton of $X$ is $\Gamma^{n-2}(\mathscr{C}(W))$ with a 2-cell attached to every 4-cycle.. 
Let $P$ be the face poset of $\mathscr{C}(W)$, and let $P_0$ be the ideal generated by $\Delta_0$. 
We can view $P_0^*$ as an upper ideal in $P^*$
 and use the fact that $(P \backslash P_0)^* = P^* \backslash P_0^*$ to get a combinatorial description of the face poset of $X$. 
The face poset of $P^*$ corresponds to the set of all cosets $wW_I$ of parabolic subgroups of $W$, 
ordered by inclusion, and $P_0^*$ is the upper order ideal whose minimal elements are cosets of irreducible parabolic subgroups of rank $2$. 
Thus, the $1$-skeleton of $X$ is clearly isomorphic to $\Gamma^{n-2}(\mathscr{C}(W))$, 
and the $2$-cells correspond to cosets of $W_{s,t}$, where $s,t \in S, m(s,t) = 2$, which are bounded by $4$-cycles. 
Hence the $2$-skeleton of $X$ is given by attaching $2$-cells to the $4$-cycles of $\Gamma^{n-2}(\mathscr{C}(W))$, 
which is the space $X_{\Gamma}$ constructed in Proposition \ref{comptograph}. 
So we obtain the following chain of isomorphisms:

$$\pi_1(\mathcal{M}(\mathscr{W}_{n,3})) \cong \pi_1(|\mathscr{C}(W)| \backslash |\Delta|) \cong \pi_1(X) \cong A_1^{n-2}(\mathscr{C}(W)).$$ 

where the final isomorphism comes from Propisition \ref{comptograph} \end{proof}

\section{An algebraic description of $\pi_1(\mathcal{M}(\mathscr{W}_{n,3}))$}

In this section, we give a description of $\pi_1(\mathcal{M}(\mathscr{W}_{n,k}))$ that is similar to the idea of a pure Artin group. 
In our case, the group we consider is a (possibly infinite) Coxeter group.  
Recall that $W$ affords the following presentation: $W$ is generated by $S$ subject to the relations:

\begin{enumerate}

\item $s^2 = 1$, $\forall s \in S$ 

\item $st = ts$, $\forall s,t \in S$ such that $m(s,t) = 2$ 

\item $sts = tst$, $\forall s,t \in S$, such that $m(s,t) = 3$ 

 \hspace{3cm} \vdots 

\item[i.] $\underbrace{stst\cdots}_{i} = \underbrace{tsts\cdots}_{i}$, $\forall s,t \in S$, such that $m(s,t) = i$ 

\hspace{3cm}  \vdots 

\end{enumerate}

where of course we have no relation of the form $st\cdots = ts\cdots$ if $m(s,t) = \infty$.

If $G$ is a group generated by $S$ subject to every relation except relations of type 1, 
then $G$ is an \emph{Artin group}. There is a surjective homomorphism $\varphi: G \to W$ 
given by $\varphi(s) = s$ for all $s \in S$. The kernel of $\varphi$ is the \emph{pure} Artin group. 
As stated in the introduction, the pure Artin group is isomorphic to the fundamental group of the complement 
of the complexification of the Coxeter arrangement for $W$. The goal of this section is to give a real analogue 
of this result for the $\mathscr{W}_{n,3}$-arrangements.

\begin{figure}[htbp]
\label{fig:dynkin}
  \center{\scalebox{.5}{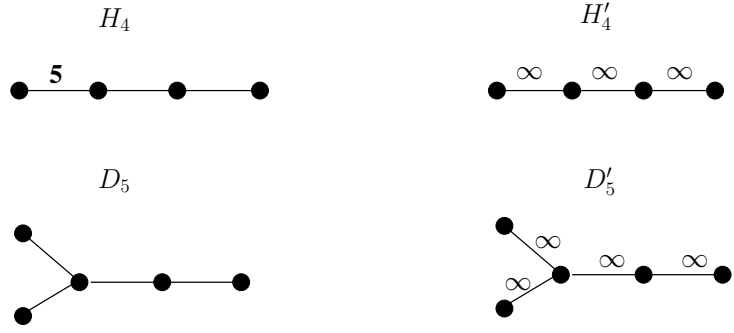}}
  \caption{Dynkin diagrams for $W$ and $W'$}
\end{figure}

 In our case, let $W'$ be a group on $S$ subject to only the relations of type 1 and 2. 
Given the Dynkin diagram $D$ for $W$, 
$W'$ is obtained by replacing all the edge labels in $D$ with the edge label $\infty$, and letting $W'$ be 
the resulting Coxeter group.

Consider the surjective homomorphism $\varphi': W' \to W$ given by $\varphi'(s) = s$ for all $s \in S$. 
Then the following result holds:

\begin{theorem}
\label{purepar}

$\pi_1(\mathcal{M}(\mathscr{W}_{n,3})) \cong \ker \varphi'$.

\end{theorem}

It is easy to see that in type $A$ and type $B$ we recover the results of Khovanov, that is $\ker \varphi'$ is the pure triplet group. Our work extends that of Khovanov to any finite real reflection group.

The proof of Theorem \ref{purepar} will appear after Lemma \ref{dilemma2}. We must first investigate the structure of $(n-2)$-loops in $\mathscr{C}(W)$ in more detail. As a result of Theorem \ref{mainresult}, we know that we can study $\pi_1(\mathcal{M}(\mathscr{W}_{n,k}))$ 
using discrete homotopy theory of $\mathscr{C}(W)$. For the duration of the section we will use the term loop to mean $(n-2)$-loop.
To any such loop $(\sigma) = (\sigma_0, \ldots, \sigma_{\ell}, \sigma_0)$ in $\mathscr{C}(W)$ we associate a sequence of elements of 
$S \cup \{1 \}$ of length $\ell+1$ in the following way: For any $i \in [\ell]$, if $\sigma_i = \sigma_{i-1}$, let $s_i = 1$. 
Otherwise let $s_i$ be the unique element $s$ of $S$ for which $\sigma_{i-1}s = \sigma_i$. Recall that $S^*$ denotes the set of all words using letters from $S$. We associate a word $f(\sigma)$ in $S^*$ to $(\sigma)$: the concatenation of the elements of the corresponding sequence in order.

We note that if $(\sigma)$ is a loop, then $f(\sigma)=1$ in $W$.
This implies that when viewing $f(\sigma)$ as a product in $W'$, $f(\sigma) \in \ker \varphi'$.
We also note that to any element $w = s_1 \cdots s_k$ in $S^*$ we can associate a chain 
$g(w) = (\sigma_0, \sigma_0s_1, \ldots, \sigma_0s_1\cdots s_m)$, where the elements $s_1 \cdots s_i$ are being viewed as elements of $W$.
If $w = 1$ when viewed as an element of $W$, then $g(w)$ is actually a loop. It is easy to see that for two loops 
$(\sigma), (\tau)$, $f((\sigma)*(\tau))= f(\sigma)f(\tau)$, and if $u,v \in S^*$, $u = v = 1$ in $W$, then $g(uv) = g(u)*g(v)$.


We now list three changes that preserve the homotopy class of a loop. We also give the effect of the changes on the corresponding word and finally we show that all homotopy grids may be constructed using only these changes. Let $e$ be the identity element.

\begin{enumerate}
\item[(T1)] Repeating simplices. By definition, the discrete homotopy equivalence relation allows us to repeat a simplex multiple times. 

If $(\sigma)$ and $(\tau)$ differ by an operation of type (T1), then $f((\sigma)) = f((\tau))$ as words, 
and thus are equal in $W'$. 



\item[(T2)] Inserting or removing a simplex. On one row there are three adjacent 
identical simplices $\alpha$, and on the bottom row the middle simplex of this triple is replaced with a new simplex 
$\beta$ that is $(n-2)$-near $\alpha$.

\begin{center}\includegraphics[height=2cm]{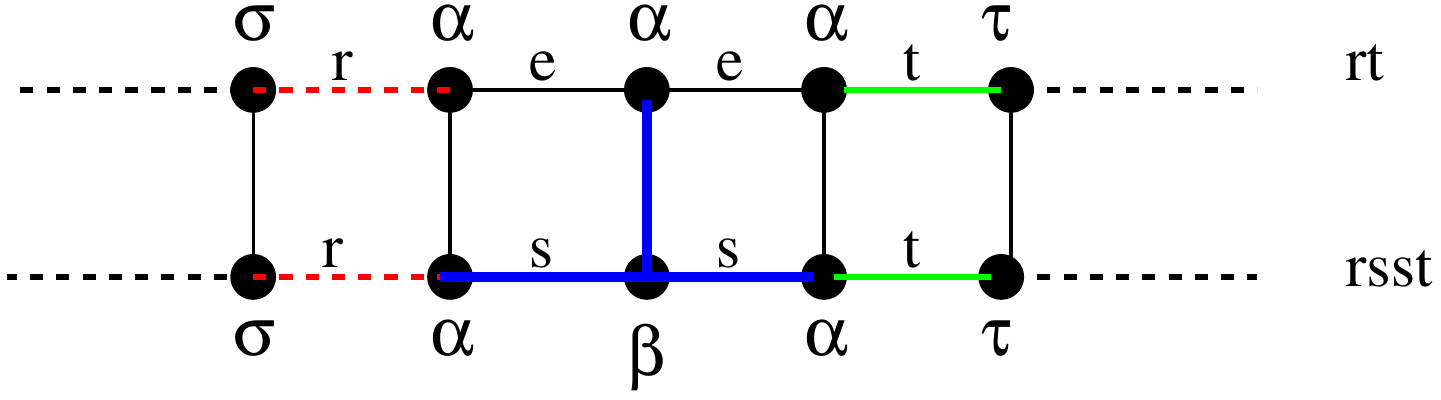}\end{center}

If $(\sigma)$ and $(\tau)$ differ by an operation of type (T2), then $f((\sigma)), f((\tau))$ differ by the insertion or removal of $s^2$ for some $s \in S$. Thus, these words are equivalent in $W'$.

\item[(T3)] Exchanging pairs that are $(n-2)$-near. Let $(\alpha, \beta, \tau, \gamma)$ be a loop of distinct simplices. We construct a discrete homotopy as shown in the figure. We note that the resulting words differ by an application of a commutative relation. It is also worth noting that this operation can only be performed when $s, t$ commute. 

\begin{center}\includegraphics[height=2cm]{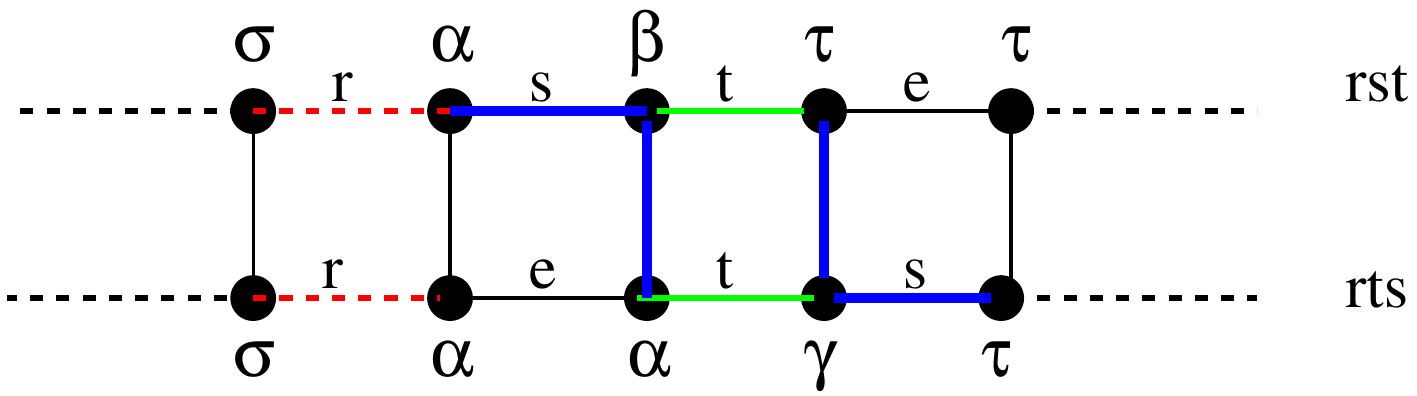}\end{center}

If $(\sigma)$ and $(\tau)$ differ by an operation of type (T3), then $f((\sigma)), f((\tau))$ 
differ by a commutative relation, and hence are equivalent in $W'$.

\end{enumerate}

\begin{proposition}
 \label{t123prop}
If two loops $(\sigma)$ and $(\tau)$ are homotopic then we may deform $(\sigma)$ into $(\tau)$ using only changes (T1)-(T3). 
\end{proposition}

\begin{proof}[Proof of Proposition \ref{t123prop}.]

By Proposition \ref{comptograph} and because $\Gamma^{n-2}(\mathcal{C}(W))$ is triangle free, we know that two loops are homotopic if and only if they differ by 4-cycles only. It is also easy to see that if $(\sigma)$ and $(\tau)$ differ by multiple 4-cycles that there is a series of intermediate loops such that each loop differs from the next by only a single 4-cycle. If we consider the loops in $\Gamma^{n-2}(\mathcal{C}(W))$, we may start with $(\sigma)$ and stretch it around the first 4-cycle where $(\sigma)$ and $(\tau)$ differ. This produces a new loop, $(\sigma_1)$ that is homotopic to $(\sigma)$. We repeat this process now with $(\sigma_1)$, stretching it around the first 4-cycle where it differs from $(\tau)$. In this way we construct a series of loops, all homotopic to one another and differing from one another in sequence by a single 4-cycle only. Thus it will suffice to consider cases where $(\sigma)$ and $(\tau)$ differ by only a single 4-cycle.

\begin{figure}[htbp]
\label{fig:dgen4cycle}
  \center{\scalebox{.68}{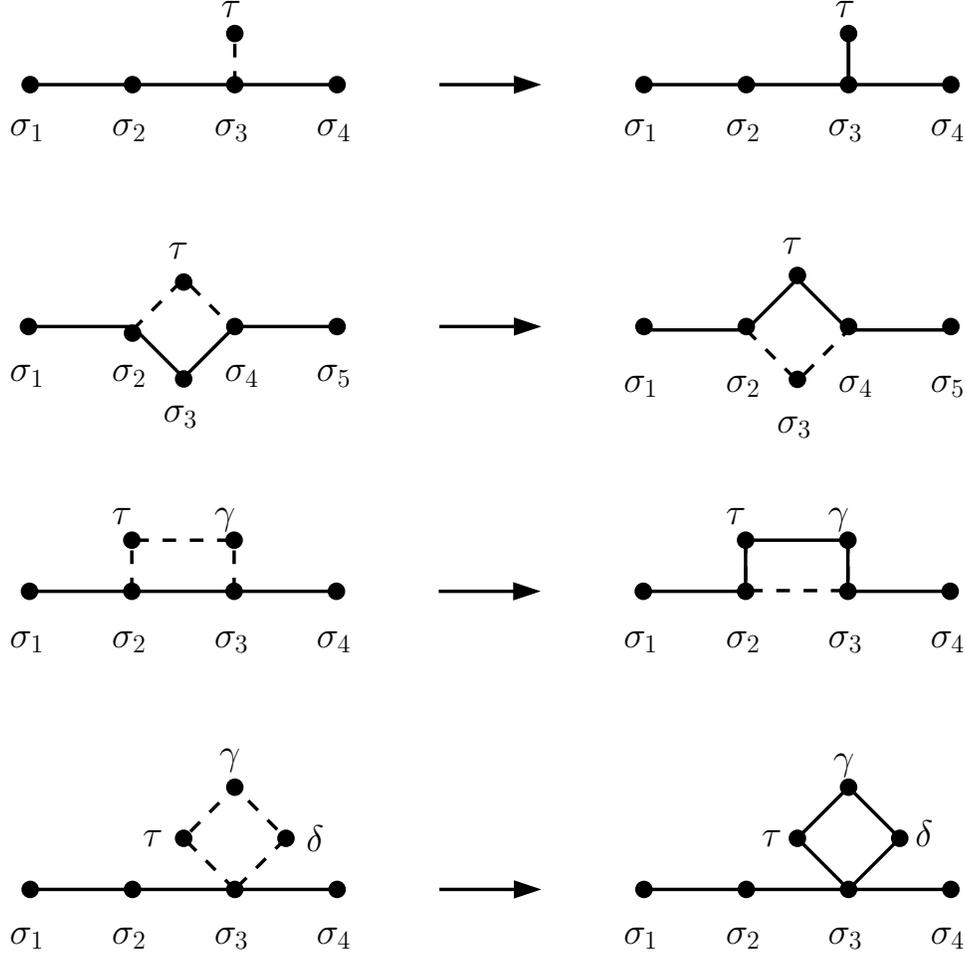}}
  \caption{The four ways in which $(\sigma)$ and $(\tau)$ may differ by a 4-cycle.}
\end{figure}

\begin{enumerate}
\item[Case 1.] $(\sigma)$ and $(\tau)$ differ by a degenerate 4-cycle. That is, the loop $(\tau)$ traverses an edge to a new simplex and then immediately follows the same edge back. See the top case in Figure \ref{fig:dgen4cycle}. This corresponds to the change (T2). \\

\item[Case 2.] $(\sigma)$ and $(\tau)$ differ by a non-degenerate 4-cycle as in the second case of Figure \ref{fig:dgen4cycle}. This is clearly the same as the change (T3) since it involves replacing two edges in a 4-cycle with the opposite two edges. 

\item[Case 3.] The loops differ as in the third case of Figure \ref{fig:dgen4cycle}. Using change (T2) to insert the simplex $\tau$ and then change (T3) to switch to the new path using $\gamma$ produces the desired homotopy. 

\item[Case 4.] This case may also be accounted for using changes (T1)-(T3). We insert both $\tau$ and $\gamma$ using change (T1). Then we change the path $\gamma$, $\tau$, $\sigma_2$ to the path $\gamma$, $\delta$, $\sigma_2$ using (T3). 

\end{enumerate}
\end{proof}

Thus for any discrete homotopy operation, the corresponding words are either equal, or differ by one of the generating relations of $W'$.

\begin{corollary}
\label{dilemma}

Let $(\sigma), (\tau)$ be loops. If $(\sigma) \simeq (\tau)$ then $f(\sigma)=f(\tau)$ in $W'$.

\end{corollary}

We want to prove a converse of Corollary \ref{dilemma}. To do that, 
we recall the concept of \emph{nil moves} and \emph{braid moves}. 
A nil move in a word $w$ is the insertion or removal of $s^2$ for some $s \in S$. A braid move in 
$w$ is the replacement of an occurrence of $\underbrace{stst\cdots}_{m(s,t)}$ with an occurrence of 
$\underbrace{tsts\cdots}_{m(s,t)}$. The following is known:

\begin{theorem}[Theorem 3.3.1 in \cite{bjorner-brenti}] 
\label{brentithm}
Let $W$ be a Coxeter group.
An expression for $w$ can be transformed into a reduced expression for $w$ by a sequence of nil moves and braid moves.
\end{theorem}

In our case, it is worth noting that the only braid moves are commutative relations.

\begin{lemma}
\label{dilemma2}
Let $w, v \in S^*$, then the following are true. 
\begin{enumerate}
\item If $w = 1$ in $W'$, then $g(w)$ is contractible.

\item If $w = v$ in $W'$, then $g(w) \simeq g(v)$.

\end{enumerate}

\end{lemma} 

\begin{proof}

Proof of 1. Since $w = 1$, the word $w$ can be transformed into the word $1$ by a sequence of nil moves 
and braid moves of length $\ell$. Proof by induction on $\ell$, where the case $\ell = 0$ is clear.

Consider the first move in the sequence. Let $w'$ be the word obtained by performing the first move in the sequence. 
If the first move is a nil move, a homotopy operation (T2) gives a discrete 
homotopy between $g(w)$ and $g(w')$ (or more accurately, a discrete homotopy between stretchings of 
$g(w), g(w')$). If the first move is a braid move, then a homotopy operation (T3) gives a discrete homotopy 
between $g(w)$ and $g(w')$. Now by induction, $g(w') \simeq (\sigma_0)$, so $g(w) \simeq (\sigma_0)$.

Proof of 2. We know that $wv^{-1} = 1$ in $W'$, so by the first part, 
$g(wv^{-1})$ is contractible, so there is a homotopy grid between $g(wv^{-1})$ and the trivial loop.
If we append $g(v)$ to the end of every row of the homotopy grid, 
we see that $g(wv^{-1})*g(v) \simeq g(v)$. Yet clearly 
\[g(w) \simeq g(wv^{-1}v) = g(wv^{-1})*g(v),\] so the result follows. \end{proof}

We may now prove the main result of this section. 

\begin{proof}[\textbf{Proof of Theorem \ref{purepar}}]

By Theorem \ref{mainresult}, it suffices to show that $A_1^{n-2}(\mathscr{C}(W)) \cong \ker \varphi'$. 
We define $F: A_1^{n-2}(\mathscr{C}(W)) \to \ker \varphi'$ as follows: 
let $[\sigma] \in A_1^{n-2}(\mathscr{C}(W))$ with representative $(\sigma)$, 
then $F([\sigma]) = f((\sigma))$. By Lemma \ref{dilemma}, this map is well-defined. 
Since $f((\sigma)*(\tau)) = f(\sigma)f(\tau)$ for any loops $(\sigma), (\tau)$, the map $F$ is clearly a homomorphism.

We define a map $G: \ker \varphi' \to A_1^{n-2}(\mathscr{C}(W))$ as follows: 
given $w \in \ker \varphi'$ with expression $w = s_1 s_2 \cdots s_k$, 
let $G(w) = [g(s_1\cdots s_k)]$. By part 2 of Lemma \ref{dilemma2}, the map $G$ is also well-defined.

Given a word $w \in S^*$, we observe that $f(g(w)) = w$. Also, 
given a loop $(\sigma)$, $g(f(\sigma)) \simeq (\sigma)$, since $(\sigma)$ 
is a stretching of $g(f(\sigma))$. Thus $F$ and $G$ are isomorphisms, and inverses of each other. \end{proof}

\section{Homotopy groups of $\mathcal{M}(\mathscr{W}_{n,k})$ for $k > 3$}

We now prove Theorem \ref{mainresult} for the case $k > 3$. 
The following result of Bj\"{o}rner and Welker will be useful.

\begin{proposition}[Corollary 3.2(ii) in \cite{bjorner-welker}]

Let $\mathscr{A}$ be a real subspace arrangement, and let $c$ be the minimum codimension of the 
subspaces of $\mathscr{A}$. Then $\pi_i(\mathcal{M}(\mathscr{A}))$ is trivial for $i < c-2$.

\end{proposition}

Since every subspace in the $k$-parabolic arrangement has codimension $k-1$, 
this proposition allows us to conclude that $\pi_1(\mathcal{M}(\mathscr{W}_{n,k}))$ 
is trivial for $k > 3$. So to finish the proof of Theorem \ref{mainresult}, it suffices to show 
that $A_1^{n-k+1}(\mathscr{C}(W))$ is trivial for $k > 3$. We fix $k > 3$ and for the rest of this section, let the base simplex, $\sigma_0$, correspond to the identity element. Also, let $W', \varphi': W' \to W$, and $G: \ker \varphi' \to A_1^{n-2}(\mathscr{C}(W))$ be as in the previous section.

We first identify $\ker \varphi'$ as being generated by the normal closure of the 
set $\{ (s,t)^{m(s,t)} : s,t \in S, m(s,t) > 2 \}$. By Theorem \ref{purepar}, we can idenfity 
$A_1^{n-2}(\mathscr{C}(W))$ in the same way. Since any $(n-2)$-loop is also an $(n-k+1)$-loop, 
we have a homomorphism $\theta : A_1^{n-2}(\mathscr{C}(W)) \to A_1^{n-k+1}(\mathscr{C}(W))$ given by sending 
the equivalence class of an $(n-2)$-loop $(\sigma)$ for the equivalence relation $\simeq_{n-2}$ 
to its equivalence class for the equivalence relation $\simeq_{n-k+1}$.

 We claim the following:

\begin{proposition}

The map $\theta \circ G$ is a surjective homomorphism. Moreover, for any word $w \in \ker \varphi'$, $\theta \circ G
(w) = [\sigma_0]$.
. 

\end{proposition}

\begin{proof}

First we show that any $(n-k+1)$-loop is discrete $(n-k+1)$-homotopic to an
$(n-k+1)$-loop that consists only of maximal simplices. Let $(\sigma) =
(\sigma_0, \sigma_1, \ldots, \sigma_k, \sigma_0)$ be an $(n-k+1)$-loop. For
each $i$, let $\tau_i$ be some maximal simplex such that $\sigma_i \subseteq
\tau_i$. The define a discrete $(n-k+1)$-homotopy grid with only two rows,
the first row being $(\sigma)$, and the second row being $(\sigma_0, \tau_1,
\ldots, \tau_k, \sigma_0)$. Clearly this grid shows that $(\sigma)$ is
$(n-k+1)$-homotopic to a loop that contains only maximal simplices. For the
rest of the proof, we will assume all simplices that appear are maximal.

Before proving surjectivity, we show that for any pair $(\sigma, \tau)$ that is $(n-k+1)$-near 
but is not $(n-2)$-near, there exists an $(n-2)$-chain $(\sigma)'$ that starts with $\sigma$ and ends 
with $\tau$, such that there is a discrete $(n-k+1)$-homotopy between $(\sigma, \tau)$ and $(\sigma)'$. 
Since $(\sigma, \tau)$ are $(n-k+1)$-near, the corresponding vertices $u, v$ on the $W$-Permutahedron 
lie on a face $F$ of dimension at most $k-2$. There is a path from $u$ to $v$ using only vertices 
$v_1, \ldots, v_j$ from $F$, with $u = v_1, v = v_j$. Let $\sigma_i$ be the simplex corresponding to $v_i$. 
Then the homotopy grid can be constructed by letting row $i$ be $(\sigma_1, \ldots, \sigma_i, \tau, \ldots \tau)$ 
where there are $j - i$ occurrences of $\tau$. Then the first row is a stretching of $(\sigma, \tau)$, 
and the $n$th row is an $(n-2)$-chain starting with $(\sigma)$ and ending with $(\tau)$, 
and by inspection the grid is a discrete $(n-k+1)$-homotopy grid.

Let $(\sigma) = (\sigma_0, \ldots, \sigma_k, \sigma_0)$ be a $(n-k+1)$-loop. 
By the previous observation, for any index $i$, if $(\sigma_i, \sigma_{i+1})$ are not 
$(n-2)$-near, we can replace $(\sigma)$ with a $(n-k+1)$-homotopic loop $(\sigma)'$ that replaces the edge
 $(\sigma_i, \sigma_{i+1})$ with an $(n-2)$-chain starting with $\sigma_i$ and ending with $\sigma_{i+1}$.
 Thus every $(n-k+1)$-loop is $(n-k+1)$-homotopic to some $(n-2)$-loop, so $\theta$ is surjective,
 and hence $\theta \circ G$ is surjective.

Consider $s,t \in S$ with $m(s,t) > 3$, and let $w = (st)^{m(s,t)}$. 
We need to show that $g(w)$ is equivalent to the trivial loop under the equivalence relation 
$\simeq_{n-k+1}$. The corresponding simplices all share an $(n-2)$-dimensional boundary, 
so the simplices $\sigma_0, \sigma_0s, \sigma_0st, \ldots, \sigma_0(st)^{m(s,t)}= \sigma_0$ are $(n-k+1)$-near. 
We construct the following homotopy grid:
The first row consists of $(\sigma)$, and the $i$th row is given by taking the 
$(i-1)$-th row, and replacing the $i$th column entry with $\sigma_0$.
It is routine to check that this $(n-k+1)$-homotopy grid shows that $g(w)$ is contractible.

Now let $u \in W'$. Let $\tau_0 = \sigma_0 u$. Then the above argument gives a $(n-k+1)$-homotopy grid 
between $(\tau_0, \tau_0 s, \ldots, \tau_0 (st)^{m(s,t)})$ and $(\tau_0)$. 
If we append the chain $g(u)$ to the start of every row of the grid, 
and $g(u^{-1})$ to the end of every row of the grid, then we see that $g(uwu^{-1}) \simeq_{n-k+1} g(uu^{-1}) \simeq_{n-k+1} (\sigma_0)$.

Since $g(w)$ is contractible for any $w$ in the normal closure of $\{(st)^{m(s,t)}: s,t \in S\}$, 
it follows that $\ker (\theta \circ G) = \ker \varphi'$. \end{proof}

\begin{corollary}

The group $A_1^{n-k+1}(\mathscr{C}(W))$ is trivial.

\end{corollary}

\section{Conclusion and Open Problems}

First, we shall discuss the $K(\pi,1)$ problem.
It follows as a result of Bj\"orner and Welker (Corollary 5 in \cite{bjorner-welker}) that for $k > 3$, the $k$-parabolic arrangements are not $K(\pi, 1)$.
However, the $\mathscr{W}_{n,3}$ -arrangement is a $K(\pi,1)$-arrangement. As a result of Davis, Januszkiewicz and Scott,
if $\mathscr{A}$ is any collection of codimension 2 subspaces of $\mathscr{H}(W )$ that are invariant under the action of $W$ ,
then $\mathscr{A}$ has a complement with trivial higher homotopy groups (Theorem 0.1.9 in \cite{blowup}). These include the $3$-parabolic arrangements.
Combined with our results, we can say that $3$-parabolic arrangements are $K(\pi, 1)$-arrangements, where $\pi = \ker \varphi'$, where $\varphi'$ 
was defined in Section 4.

However, this motivates the following natural question. Let $\mathscr{A}$ be a collection of codimension 2 subspaces that is invariant under the 
action of a finite reflection group $W$. Since the complement has trivial higher homotopy groups, $\mathscr{A}$ is a $K(\pi,1)$-arrangement for some group $\pi$ - can we 
give a presentation for $\pi$? Based upon our results, and similar results of Khovanov \cite{khovanov}, we have the following conjecture for a presentation 
for $\pi$. First, note that under our correspondence, a collection of codimension 2 subspaces that are invariant under the group action corresponds to 
a collection of rank 2 parabolic subgroups that is invariant under conjugation.

\begin{conjecture}

Let $\mathscr{P}$ be a collection of rank $2$ parabolic subgroups of a finite real reflection group $W$ 
such that $\mathscr{P}$ is closed under conjugation, and let $\mathscr{W} = \{Fix(G) : G \in \mathscr{P} \}$ 
Define a new Coxeter group $W'$ with the same generating set $S$ as $W$, and subject to:

$m'(s,t) = \left\{ \begin{array}{ll} \infty & \mbox{ if } <s,t> \in \mathscr{P} \\ m(s,t) & \mbox{ else} \end{array} \right.$

and let $\varphi' : W' \to W$ be given by sending $s \to s$ for all $s \in S$. Then $\pi_1(\mathcal{M}(\mathscr{W})) \cong \ker \varphi'$.

Moreover, $\mathscr{W}$ is a $K(\pi,1)$-arrangement, where $\pi = \ker \varphi'$.

\end{conjecture}

Now we shall discuss some other open questions involving discrete homotopy theory and subspace arrangements.
In \cite{foundations}, Barcelo, et al., give a definition for higher discrete homotopy groups, which we denote $A_m^q(\Delta, \sigma_0)$.
A natural question is whether or not these groups are related to the higher homotopy groups of $\mathcal{M}(\mathscr{W}_{n,k})$.

\begin{conjecture}
\label{highertheory}
Let $\mathcal{M}(\mathscr{W}_{n,k})$ be the complement of the $k$-parabolic arrangement $\mathscr{W}_{n,k}$.

Then $\pi_m(\mathcal{M}(\mathscr{W}_{n,k})) \cong A_m^{n-k+1}(\mathscr{C}(W))$.

\end{conjecture}

For $m < k$, it would suffice to show that $A_m^{n-k+1}(\mathscr{C}(W))$ is trivial. 
The conjecture becomes interesting for $k > 3, m = k$, because in this case the $k$-th homology 
group of $\mathcal{M}(\mathscr{W}_{n,k})$ is isomorphic to the $k$-th homotopy group. 
Thus, one could find the formulas for the first non-zero Betti numbers using discrete homotopy theory. 
Determining the Betti numbers for the $k$-parabolic arrangements is also an open problem, 
in the case that $W$ is an exceptional group.

Finally, one may wonder if it is possible to generalize Theorem \ref{mainresult} to other hyperplane arrangements. 
That is, given a hyperplane arrangement $\mathscr{H}$, let $\mathscr{C}(\mathscr{H})$ be the face complex of $\mathscr{H}$. 
Is there a subspace arrangement $\mathscr{A}$ for which $\pi_1(\mathcal{M}(\mathscr{A})) \cong A_1^{n-2}(\mathscr{C}(\mathscr{H}))$? 
This would be an example of using discrete homotopy theory of a complex that arises from geometry to study a topological space 
related to the original complex. 

\section{Acknowledgements}

We would like to thank Nathan Reading, Volkmar Welker, and Eric Babson for helpful conversations. We
would also like to thank Richard Scott for bringing the results in \cite{blowup} to our attention, which answers the question about whether or not the $3$-parabolic arrangement is $K(\pi,1)$.

\bibliographystyle{amsplain}

\bibliography{k_parbib}

\providecommand{\bysame}{\leavevmode\hbox to3em{\hrulefill}\thinspace}
\providecommand{\MR}{\relax\ifhmode\unskip\space\fi MR }
\providecommand{\MRhref}[2]{%
  \href{http://www.ams.org/mathscinet-getitem?mr=#1}{#2}
}
\providecommand{\href}[2]{#2}
\begin{thebibliography}{10}

\bibitem{arnold}
Vladimir~I. Arnold, \emph{The cohomology ring of the group of dyed braids},
  1969, pp.~227--231. \MR{MR0242196 (39 \#3529)}

\bibitem{arvola}
William~A. Arvola, \emph{The fundamental group of the complement of an
  arrangement of complex hyperplanes}, Topology \textbf{31} (1992), no.~4,
  757--765. \MR{MR1191377 (93k:32078)}

\bibitem{graph-homotopy}
Eric Babson, H{\'e}l{\`e}ne Barcelo, Mark {de Longueville}, and Reinhard
  Laubenbacher, \emph{Homotopy theory of graphs}, J. Algebraic Combin.
  \textbf{24} (2006), no.~1, 31--44. \MR{MR2245779 (2007d:05156)}

\bibitem{barcelo-ihrig}
H{\'e}l{\`e}ne Barcelo and Edwin Ihrig, \emph{Lattices of parabolic subgroups
  in connection with hyperplane arrangements}, J. Algebraic Combin. \textbf{9}
  (1999), no.~1, 5--24. \MR{MR1676736 (2000g:52023)}

\bibitem{foundations}
H{\'e}l{\`e}ne Barcelo, Xenia Kramer, Reinhard Laubenbacher, and Christopher
  Weaver, \emph{Foundations of a connectivity theory for simplicial complexes},
  Adv. in Appl. Math. \textbf{26} (2001), no.~2, 97--128. \MR{MR1808443
  (2001k:57029)}

\bibitem{perspectives}
H{\'e}l{\`e}ne Barcelo and Reinhard Laubenbacher, \emph{Perspectives on
  {$A$}-homotopy theory and its applications}, Discrete Math. \textbf{298}
  (2005), no.~1-3, 39--61. \MR{MR2163440 (2006f:52017)}

\bibitem{barcelo-smith}
H{\'e}l{\`e}ne Barcelo and Shelly Smith, \emph{The discrete fundamental group
  of the order complex of {$B\sb n$}}, J. Algebraic Combin. \textbf{27} (2008),
  no.~4, 399--421. \MR{MR2393249}

\bibitem{bjorner-atheory}
Anders Bj{\"o}rner, \emph{Personal communication}, 1999.

\bibitem{bjorner-brenti}
Anders Bj{\"o}rner and Francesco Brenti, \emph{Combinatorics of {C}oxeter
  groups}, Graduate Texts in Mathematics, vol. 231, Springer, New York, 2005.
  \MR{MR2133266 (2006d:05001)}

\bibitem{subspacesBD}
Anders Bj{\"o}rner and Bruce~E. Sagan, \emph{Subspace arrangements of type
  {$B\sb n$} and {$D\sb n$}}, J. Algebraic Combin. \textbf{5} (1996), no.~4,
  291--314. \MR{MR1406454 (97g:52028)}

\bibitem{bjorner-welker}
Anders Bj{\"o}rner and Volkmar Welker, \emph{The homology of ``{$k$}-equal''
  manifolds and related partition lattices}, Adv. Math. \textbf{110} (1995),
  no.~2, 277--313. \MR{MR1317619 (95m:52029)}

\bibitem{bjorner-ziegler}
Anders Bj{\"o}rner and G{\"u}nter~M. Ziegler, \emph{Combinatorial
  stratification of complex arrangements}, J. Amer. Math. Soc. \textbf{5}
  (1992), no.~1, 105--149. \MR{MR1119198 (92k:52022)}

\bibitem{brieskorn}
Egbert Brieskorn, \emph{Die {F}undamentalgruppe des {R}aumes der regul{\"a}ren
  {O}rbits einer endlichen komplexen {S}piegelungsgruppe}, Invent. Math.
  \textbf{12} (1971), 57--61. \MR{MR0293615 (45 \#2692)}

\bibitem{blowup}
Michael Davis, Tadeusz Januszkiewicz, and Richard Scott, \emph{Nonpositive
  curvature of blow-ups}, Selecta Math. (N.S.) \textbf{4} (1998), no.~4,
  491--547. \MR{MR1668119 (2001f:53078)}

\bibitem{deligne}
Pierre Deligne, \emph{Les immeubles des groupes de tresses
  g{\'e}n{\'e}ralis{\'e}s}, Invent. Math. \textbf{17} (1972), 273--302.
  \MR{MR0422673 (54 \#10659)}

\bibitem{fadell-neuwirth}
Edward Fadell and Lee Neuwirth, \emph{Configuration spaces}, Math. Scand.
  \textbf{10} (1962), 111--118. \MR{MR0141126 (25 \#4537)}

\bibitem{subspacesD}
Eva~Maria Feichtner and Dmitry~N. Kozlov, \emph{On subspace arrangements of
  type {$D$}}, Discrete Math. \textbf{210} (2000), no.~1-3, 27--54, Formal
  power series and algebraic combinatorics (Minneapolis, MN, 1996).
  \MR{MR1731606 (2001k:52039)}

\bibitem{fox-neuwirth}
Ralph Fox and Lee Neuwirth, \emph{The braid groups}, Math. Scand. \textbf{10}
  (1962), 119--126. \MR{MR0150755 (27 \#742)}

\bibitem{goresky-macpherson}
Mark Goresky and Robert MacPherson, \emph{Stratified {M}orse theory},
  Ergebnisse der Mathematik und ihrer Grenzgebiete (3) [Results in Mathematics
  and Related Areas (3)], vol.~14, Springer-Verlag, Berlin, 1988. \MR{MR932724
  (90d:57039)}

\bibitem{coxbook}
James~E. Humphreys, \emph{Reflection groups and {C}oxeter groups}, Cambridge
  Studies in Advanced Mathematics, vol.~29, Cambridge University Press,
  Cambridge, 1990. \MR{MR1066460 (92h:20002)}

\bibitem{khovanov}
Mikhail Khovanov, \emph{Real {$K(\pi,1)$} arrangements from finite root
  systems}, Math. Res. Lett. \textbf{3} (1996), no.~2, 261--274. \MR{MR1386845
  (97d:52023)}

\bibitem{orlik-terao}
Peter Orlik and Hiroaki Terao, \emph{Arrangements of hyperplanes}, Grundlehren
  der Mathematischen Wissenschaften [Fundamental Principles of Mathematical
  Sciences], vol. 300, Springer-Verlag, Berlin, 1992. \MR{MR1217488
  (94e:52014)}

\bibitem{polybook}
G{\"u}nter~M. Ziegler, \emph{Lectures on polytopes}, Graduate Texts in
  Mathematics, vol. 152, Springer-Verlag, New York, 1995. \MR{MR1311028
  (96a:52011)}

\bibitem{zieg-ziv}
G{\"u}nter~M. Ziegler and Rade~T. {\v{Z}}ivaljevi{\'c}, \emph{Homotopy types of
  subspace arrangements via diagrams of spaces}, Math. Ann. \textbf{295}
  (1993), no.~3, 527--548. \MR{MR1204836 (94c:55018)}

\end{thebibliography}

\end{document}